\documentclass[11pt, reqno]{amsart}
\usepackage[utf8]{inputenc}
\usepackage{amssymb, amsmath, amsthm, mathrsfs, setspace, enumerate}
\usepackage[colorlinks=true,linkcolor=purple, citecolor=blue,urlcolor=magenta]{hyperref}
\usepackage{amsmath, amssymb}
\usepackage{tikz}
\usepackage{pgfplots}
\usepackage{xcolor}
\usepackage{graphicx}
\usepackage{float}  
\newcommand{\coloneqq}{\mathrel{:=}}

\newtheorem{theo}{Theorem}[section]
\newtheorem{lem}{Lemma}[section]

\numberwithin{equation}{section}
\title[Sharp Hankel and Hermitian--Toeplitz Determinants...]{Sharp Hankel and Hermitian--Toeplitz Determinants Involving Logarithmic and Inverse Logarithmic Coefficients for the class \(\mathcal{S}_{car}^*\)}
\author[N. Sarkar and P. Das]{Nabadwip Sarkar and Pradip Das}
\address{Amity School of Applied Sciences, Amity University Mumbai, Panvel, Navi Mumbai, Maharashtra-410206, India}
\email{nsarkar@mum.amity.edu, naba.iitbmath@gmail.com}
\address{Department of Mathematics, Raiganj University, Raiganj, West Bengal-733134, India.}
\email{pradipsmath@gmail.com}

\begin{document}

\maketitle
\renewcommand{\thefootnote}{}
\footnote{2020 \emph{Mathematics Subject Classification}:30C45, 30C50}
\footnote{\emph{Key words and phrases}: Univalent functions, cardiod domain, Inverse Logarithmic coefficients, Hankel determinant, Hermitian--Toeplitz Determinants}
\footnote{*\emph{Corresponding Author}: Pradip Das.}

\renewcommand{\thefootnote}{\arabic{footnote}}
\setcounter{footnote}{0}
\begin{abstract}
This paper focuses on coefficient problems for the class of starlike functions related to a cardioid region. We derive sharp bounds for the Hankel determinants of logarithmic coefficients and logarithmic inverse coefficients of order two. Additionally, we completely solve the extremal problem for the third-order Hermitian--Toeplitz determinant by providing its sharp upper and lower bounds. All results are supported by extremal examples demonstrating the sharpness of the obtained estimates.
\end{abstract}
\section{Introduction}

Let \(\mathcal{H}\) represent the class of analytic functions defined in the open unit disk \(\mathbb{D}\coloneqq \{z\in \mathbb{C}:|z|<1\}\). The space \(\mathcal{H}\) is a locally convex topological vector space equipped with the topology of uniform convergence on compact subsets of \(\mathbb{D}\). Define \(\mathcal{A}\) as the subclass of functions \(f\in \mathcal{H}\) satisfying \(f(0)=0\) and \(f'(0)=1\). Further, let \(\mathcal{S}\) denote the subset of \(\mathcal{A}\) consisting of functions that are univalent (one-to-one) in \(\mathbb{D}\). For any \(f\in \mathcal{A}\), it admits the series representation
\begin{equation}
f(z) = z + \sum_{n=2}^{\infty} a_n z^n, \quad z\in \mathbb{D}. \label{eq:series}
\end{equation}

For \(q,n\in \mathbb{N}\), the Hankel determinant \(H_{q,n}(f)\) associated with the Taylor coefficients of a function \(f\in \mathcal{A}\), as given in \eqref{eq:series}, is defined as
\[
H_{q,n}(f) = 
\begin{vmatrix}
a_n & a_{n+1} & \dots & a_{n+q-1} \\
a_{n+1} & a_{n+2} & \dots & a_{n+q} \\
\vdots & \vdots & \ddots & \vdots \\
a_{n+q-1} & a_{n+q} & \dots & a_{n+2q-2}
\end{vmatrix}.
\]

Hankel determinants of various orders have been extensively studied in different contexts (see, for example, \cite{ref4}). A notable case is the second Hankel determinant, \(H_{2,1}(f)\), which is also known as the Fekete--Szeg\H{o} functional. Fekete and Szeg\H{o} derived bounds for \(|a_3 - \mu a_2^2|\) where \(\mu\) is a real parameter (see \cite[Theorem 3.8]{ref11}). For a function \(f\in \mathcal{S}\), let \(g\) denote its inverse, defined in a neighborhood of the origin with the Taylor series expansion
\begin{equation}
g(w) = f^{-1}(w) = w + \sum_{n=2}^{\infty} A_n w^n, \label{eq:invseries}
\end{equation}
where the radius of convergence can be chosen as \(|w|<1/4\) by Koebe's \(1/4\)-theorem. Using variational methods, L\"owner \cite{ref28} established the sharp bound
\[
|A_n| \leq K_n, \quad \text{for all } n\in \mathbb{N},
\]
where \(\displaystyle K_n = \frac{(2n)!}{n!(n+1)!}\) and \(K(w)=w+K_2w^2+K_3w^3+\dots\) represents the inverse of the Koebe function. If \(f(z)=z+\sum_{n=2}^{\infty}a_n z^n\) is a member of \(\mathcal{S}\), then the relation \(f(f^{-1})(w)=w\) implies, from the series expansion \eqref{eq:invseries}, the following coefficients:
\begin{equation}
\begin{cases}
A_2 &= -a_2, \\
A_3 &= -a_3 + 2a_2^2, \\
A_4 &= -a_4 + 5a_2a_3 - 5a_2^3.
\end{cases} \label{eq:invcoeff}
\end{equation}

The logarithmic coefficients \(\gamma_n\) of a function \(f\in \mathcal{S}\) are defined by
\begin{equation}
F_f(z) \coloneqq \log \frac{f(z)}{z} = 2\sum_{n=1}^{\infty} \gamma_n z^n, \quad z\in \mathbb{D}. \label{eq:logcoeff}
\end{equation}

Few exact upper bounds for \(\gamma_n\) are known. Milin \cite{ref29} highlighted the importance of studying these bounds in connection with the Bieberbach conjecture. Milin conjectured that for \(f\in \mathcal{S}\) and \(n\geq 2\)
\[
\sum_{m=1}^{n} \sum_{k=1}^{m} \left(k|\gamma_k|^2 - \frac{1}{k}\right) \leq 0.
\]
This conjecture, proven by De Branges, ultimately led to the resolution of the Bieberbach conjecture \cite{ref5}. For the Koebe function \(k(z)=z/(1-z)^2\), the logarithmic coefficients are given by \(\gamma_n=1/n\). Since the Koebe function is the extremal function for many extremal problems in \(\mathcal{S}\), it is natural to conjecture that \(|\gamma_n|\leq 1/n\) for all \(f\in \mathcal{S}\). However, this conjecture does not hold in general, even asymptotically. For example, there exists a bounded function \(f\in \mathcal{S}\) with logarithmic coefficients \(\gamma_n \not\sim O(n^{-0.83})\) (see \cite[Theorem 8.4]{ref11}). By differentiating the logarithmic series \eqref{eq:logcoeff} and comparing coefficients, the following expressions for \(\gamma_n\) in terms of \(a_n\) are obtained:
\begin{equation}
\begin{cases}
\gamma_1 &= \frac{1}{2} a_2,\\[2mm]
\gamma_2 &= \frac{1}{2}\left(a_3 - \frac{1}{2} a_2^2\right),\\[2mm]
\gamma_3 &= \frac{1}{2}\left(a_4 - a_2a_3 + \frac{1}{3} a_2^3\right).
\end{cases} \label{eq:gamma}
\end{equation}
If \(f\in \mathcal{S}\), it follows that \(|\gamma_1|\leq 1\), since \(|a_2|\leq 2\). Using the Fekete--Szeg\H{o} inequality \cite[Theorem 3.8]{ref11} for functions in \(\mathcal{S}\) and substituting into \eqref{eq:logcoeff}, the sharp estimate for \(\gamma_2\) is given by
\[
|\gamma_2| \leq \frac{1}{2}(1+2e^{-2}) \approx 0.635.
\]
For \(n\geq 3\), deriving bounds for \(|\gamma_n|\) becomes considerably more challenging, and no significant general bounds for \(|\gamma_n|\) for functions in \(\mathcal{S}\) are currently known. In 2017, Ali and Allu \cite{ref1} provided initial bounds on logarithmic coefficients for close-to-convex functions. For recent advances on several subclasses of close-to-convex functions, see \cite{ref2, ref6, ref32}. The concept of logarithmic inverse coefficients, introduced by Ponnusamy et al. \cite{ref31}, concerns the logarithmic coefficients of the inverse function \(f^{-1}\). These coefficients, denoted by \(\Gamma_n\), are defined through the relation
\[
F_{f^{-1}}(w) \coloneqq \log \frac{f^{-1}(w)}{w} = 2\sum_{n=1}^{\infty} \Gamma_n w^n, \quad |w|<\frac{1}{4},
\]
where
\begin{equation}
\begin{cases}
\Gamma_1 &= -\frac{1}{2} a_2,\\[2mm]
\Gamma_2 &= -\frac{1}{2} a_3 + \frac{3}{4} a_2^2,\\[2mm]
\Gamma_3 &= -\frac{1}{2}\left(a_4 - 4a_2a_3 + \frac{10}{3} a_2^3\right).
\end{cases} \label{eq:invlogcoeff}
\end{equation}
Ponnusamy et al. \cite{ref31} derived sharp upper bounds for these coefficients in the class \(\mathcal{S}\) expressed as
\[
|\Gamma_n| \leq \frac{1}{2n} \binom{2n}{n}, \quad n\in \mathbb{N},
\]
with equality attained only for the Koebe function and its rotations. They also provided sharp bounds for the initial logarithmic inverse coefficients for several important geometric subclasses of \(\mathcal{S}\). Kowalczyk and Lecko \cite{ref16} recently proposed the study of the Hankel determinant with entries as logarithmic coefficients of \(f\in \mathcal{S}\), defined as
\[
H_{q,n}(F_f/2) = 
\begin{vmatrix}
\gamma_n & \gamma_{n+1} & \dots & \gamma_{n+q-1} \\
\gamma_{n+1} & \gamma_{n+2} & \dots & \gamma_{n+q} \\
\vdots & \vdots & \ddots & \vdots \\
\gamma_{n+q-1} & \gamma_{n+q} & \dots & \gamma_{n+2q-2}
\end{vmatrix}.
\]
Also the concept of the Hankel determinant \(H_{q,n}(F_{f^{-1}}/2)\) \cite{ref3} where the elements of the determinants are logarithmic coefficients of the inverse functions \(f\in \mathcal{S}\) are expressed as
\[
H_{q,n}(F_{f^{-1}}/2) = 
\begin{vmatrix}
\Gamma_n & \Gamma_{n+1} & \dots & \Gamma_{n+q-1} \\
\Gamma_{n+1} & \Gamma_{n+2} & \dots & \Gamma_{n+q} \\
\vdots & \vdots & \ddots & \vdots \\
\Gamma_{n+q-1} & \Gamma_{n+q} & \dots & \Gamma_{n+2q-2}
\end{vmatrix}.
\]

In recent years, considerable attention has been devoted to the study of Hankel determinants involving logarithmic coefficients for various subclasses of analytic functions such as starlike, convex, univalent, strongly starlike, and strongly convex functions (see \cite{ref18, ref19, ref20, ref34} and the references therein).

In a similar direction, it is worthwhile to study the problem of estimating the Hermitian--Toeplitz determinant. For \(q,n\in \mathbb{N}\), the Hermitian--Toeplitz determinant (see \cite{ref13, ref21}) of order \(n\) associated with the sequence \(\{a_k\}_{k\geq 1}\) of coefficients of a function \(f\in \mathcal{A}\) is defined as
\begin{equation}
T_{q,n}(f) \coloneqq 
\begin{vmatrix}
a_n & a_{n+1} & \dots & a_{n+q-1} \\
\overline{a_{n+1}} & a_n & \dots & a_{n+q-2} \\
\vdots & \vdots & \ddots & \vdots \\
\overline{a_{n+q-1}} & \overline{a_{n+q-2}} & \dots & a_n
\end{vmatrix}. \label{eq:HT}
\end{equation}
A direct computation from \eqref{eq:HT} leads to the following third-order Hermitian--Toeplitz determinants:
\begin{equation}
T_{3,1}(f) = 2\Re\left(a_2^2 \overline{a_3}\right) - 2|a_2|^2 - |a_3|^2 + 1. \label{eq:HT31}
\end{equation}

The investigation of Hermitian--Toeplitz determinants for various subclasses of normalized analytic functions was initially carried out in \cite{ref10, ref15} and later extended in \cite{ref9}. More recent contributions in this direction were obtained in \cite{ref25, ref17}. In particular, Cudna et al. \cite{ref10} established the sharp bounds for the second- and third-order Hermitian--Toeplitz determinants within the classes of starlike and convex functions of order \(\beta\). Subsequently, Kumar et al. \cite{ref25} determined sharp estimates for the same determinants in the setting of Janowski starlike and convex functions, thereby generalizing the results of \cite{ref10}. Additional developments and related results can be found in \cite{ref8, ref22, ref23, ref24, ref35, ref36, ref37}.

We write \(\mathcal{S}\) for the subclass of \(\mathcal{A}\) consisting of functions that are univalent in \(\mathbb{D}\). A function \(f\) is said to be subordinate to a function \(g\), denoted by \(f\prec g\), if there exists an analytic function \(w\) with \(|w(z)|\leq |z|\) and \(w(0)=0\) such that \(f(z)=g(w(z))\). If \(g\) is univalent and \(f(0)=g(0)\), then \(f(\mathbb{D})\subseteq g(\mathbb{D})\).

Consider the function \(\phi: \mathbb{D}\to \mathbb{C}\) defined by
\[
\phi(z) = 1 + \frac{4}{3}z + \frac{2}{3}z^2.
\]
The function \(\phi\) maps \(\mathbb{D}\) onto the region bounded by the cardioid
\[
\Delta_C = \left\{x+iy: (9x^2+9y^2-18x+5)^2 - 16(9x^2+9y^2-6x+1) = 0\right\}.
\]
Using the function \(\phi\), a subclass of starlike functions related to a cardioid domain was introduced by Sharma et al. \cite{ref33}. It is defined as
\[
\mathcal{S}_{car}^* \coloneqq \left\{ f\in \mathcal{A}: \frac{z f'(z)}{f(z)} \prec 1 + \frac{4}{3}z + \frac{2}{3}z^2 \right\}.
\]

In this paper, we divide our study into two sections. In the first section, we explore the concept of the second Hankel determinant involving logarithmic coefficients and logarithmic inverse coefficients for the class \(\mathcal{S}_{car}^*\). In the second section, we investigate the third Hermitian--Toeplitz determinant for the class \(\mathcal{S}_{car}^*\).

\section{Lemmas}

In this section, we present the essential lemmas that serve as the foundation for establishing the main results of this paper. Let \(\mathcal{P}\) denote the family of all analytic functions \(p\) having positive real part in \(\mathbb{D}\), which can be written as
\begin{equation}
p(z) = 1 + c_1 z + c_2 z^2 + c_3 z^3 + \cdots. \label{eq:car}
\end{equation}
Functions belonging to \(\mathcal{P}\) are commonly referred to as Carath\'{e}odory functions. It is a classical result that \(|c_n|\leq 2\) for \(n\geq 1\) whenever \(p\in \mathcal{P}\).

The parametric description of the coefficients often plays a vital role. In Lemma \ref{lem:coeff2}, equation \eqref{eq:c3} originates from the work of Carath\'{e}odory \cite{ref11}. Equation \eqref{eq:c1c2} was established in \cite{ref30}. Later, in 1982, Libera and Zlotkiewicz \cite{ref26, ref27} derived equation \eqref{eq:c3} under the restriction \(c_1\geq 0\). Subsequently, Cho et al. \cite{ref7} extended equation \eqref{eq:c3} to the general case and also provided an explicit extremal function.

\begin{lem}[{\cite[Lemma 3, p. 254]{ref27}}]\label{lem:coeff1}
Let \(\mathcal{P}\) denote the class of analytic functions with the Taylor expansion
\begin{equation}
p(z) = 1 + c_1 z + c_2 z^2 + c_3 z^3 + \cdots, \label{eq:pexp}
\end{equation}
satisfying \(\Re\{p(z)\}>0\) for \(z\in \mathbb{D}\). Then
\[
2c_2 = c_1^2 + (4 - c_1^2)\xi,
\]
for some \(\xi \in \overline{\mathbb{D}}\).
\end{lem}

\begin{lem}[\cite{ref26, ref27}]\label{lem:coeff2}
If \(p\in \mathcal{P}\) is of the form \eqref{eq:pexp} with \(c_1\geq 0\), then
\begin{align}
c_1 &= 2\tau_1, \label{eq:c1}\\
c_2 &= 2\tau_1^2 + 2(1-\tau_1^2)\tau_2, \label{eq:c1c2}
\end{align}
and
\begin{equation}
c_3 = 2\tau_1^3 + 4(1-\tau_1^2)\tau_1\tau_2 - 2(1-\tau_1^2)\tau_1\tau_2^2 + 2(1-\tau_1^2)(1-|\tau_2|^2)\tau_3, \label{eq:c3}
\end{equation}
for some \(\tau_1\in [0,1]\) and \(\tau_2,\tau_3\in \overline{\mathbb{D}}\coloneqq \{z\in \mathbb{C}: |z|\leq 1\}\).
\end{lem}

For \(\tau_1\in \mathbb{T}\coloneqq \{z\in \mathbb{C}: |z|=1\}\), there exists a unique function \(p\in \mathcal{P}\) with \(c_1\) as in \eqref{eq:c1}, namely,
\[
p(z) = \frac{1+\tau_1 z}{1-\tau_1 z}, \quad z\in \mathbb{D}.
\]
For \(\tau_1\in \mathbb{D}\) and \(\tau_2\in \mathbb{T}\), there exists a unique function \(p\in \mathcal{P}\) with \(c_1\) and \(c_2\) as in \eqref{eq:c1} and \eqref{eq:c1c2}, namely,
\[
p(z) = \frac{1+(\bar{\tau}_1\tau_2+\tau_1)z+\tau_2 z^2}{1+(\bar{\tau}_1\tau_2-\tau_1)z-\tau_2 z^2}, \quad z\in \mathbb{D}.
\]
For \(\tau_1,\tau_2\in \mathbb{D}\) and \(\tau_3\in \mathbb{T}\), there exists a unique function \(p\in \mathcal{P}\) with \(c_1\), \(c_2\), and \(c_3\) as in \eqref{eq:c1}--\eqref{eq:c3}, namely,
\[
p(z) = \frac{1+(\bar{\tau}_2\tau_3+\bar{\tau}_1\tau_2+\tau_1)z+(\bar{\tau}_1\tau_3+\tau_1\bar{\tau}_2\tau_3+\tau_2)z^2+\tau_3 z^3}{1+(\bar{\tau}_2\tau_3+\bar{\tau}_1\tau_2-\tau_1)z+(\bar{\tau}_1\tau_3-\tau_1\bar{\tau}_2\tau_3-\tau_2)z^2-\tau_3 z^3}, \quad z\in \mathbb{D}.
\]

\begin{lem}[\cite{ref9}]\label{lem:max}
Let \(A\), \(B\), \(C\) be real numbers, and let
\[
Y(A,B,C) \coloneqq \max \{ |A+Bz+Cz^2| + 1 - |z|^2 : z\in \overline{\mathbb{D}} \}.
\]
\begin{enumerate}
\item[(i)] If \(AC\geq 0\), then
\[
Y(A,B,C) = 
\begin{cases}
|A|+|B|+|C|, & \text{if } |B| \geq 2(1-|C|), \\[1.5ex]
1+|A| + \frac{B^2}{4(1-|C|)}, & \text{if } |B| < 2(1-|C|).
\end{cases}
\]
\item[(ii)] If \(AC<0\), then
\[
Y(A,B,C) = 
\begin{cases}
1-|A| + \frac{B^2}{4(1-|C|)}, & \text{if } -4AC\left(\frac{1}{C^2}-1\right) \leq B^2 \text{ and } |B|<2(1-|C|), \\[1.5ex]
1+|A| + \frac{B^2}{4(1+|C|)}, & \text{if } B^2 < \min\left\{4(1+|C|)^2, -4AC\left(\frac{1}{C^2}-1\right)\right\}, \\[1.5ex]
R(A,B,C), & \text{otherwise},
\end{cases}
\]
where
\[
R(A,B,C) = 
\begin{cases}
|A|+|B|-|C|, & \text{if } |C|(|B|+4|A|) \leq |AB|, \\[1.5ex]
-|A|+|B|+|C|, & \text{if } |AB| \leq |C|(|B|-4|A|), \\[1.5ex]
(|C|+|A|)\sqrt{1-\frac{B^2}{4AC}}, & \text{otherwise}.
\end{cases}
\]
\end{enumerate}
\end{lem}

\section{Hankel determinants involving logarithmic and inverse logarithmic coefficients}

In this section, we provide sharp bounds on the second-order Hankel determinant for the logarithmic coefficients and logarithmic inverse coefficients for the class \(\mathcal{S}_{car}^*\).

\begin{theo}
Let \(f\in \mathcal{S}_{car}^*\) be given by \eqref{eq:series}. Then
\[
\left|H_{2,1}(F_f/2)\right| \leq \frac{16}{27}.
\]
The inequality is sharp.
\end{theo}

\begin{proof}
Let \(f\in \mathcal{S}_{car}^*\). There exists a Schwarz function \(w\) with \(w(0)=0\) and \(|w(z)|<1\) in \(\mathbb{D}\) such that
\begin{equation}
\frac{z f'(z)}{f(z)} = 1 + \frac{4}{3} w(z) + \frac{2}{3} w^2(z). \label{eq:sub1}
\end{equation}
If \(p\in \mathcal{P}\), then we can write
\begin{equation}
w(z) = \frac{p(z)-1}{p(z)+1}. \label{eq:wtop}
\end{equation}
Let \(p\) be given by \eqref{eq:car}. From \eqref{eq:sub1} and \eqref{eq:wtop}, by equating the coefficients we obtain
\begin{equation}
\begin{cases}
a_2 = \frac{2}{3} c_1,\\[2mm]
a_3 = \frac{1}{3} c_2 + \frac{5}{36} c_1^2,\\[2mm]
a_4 = -\frac{1}{162} c_1^3 + \frac{1}{9} c_1 c_2 + \frac{2}{9} c_3.
\end{cases} \label{eq:coeff1}
\end{equation}
By substituting the above expressions for \(a_2, a_3\) and \(a_4\) in \eqref{eq:gamma}, a direct computation yields
\begin{equation}
\begin{aligned}
H_{2,1}(F_f/2) &= \gamma_1 \gamma_3 - \gamma_2^2 \\
&= \frac{1}{4}\left(a_2 a_4 - a_3^2 + \frac{1}{12} a_2^4\right)\\
&= \frac{1}{432} (-3c_1^4 - 8c_1^2 c_2 + 64c_1 c_3 - 48c_2^2).
\end{aligned} \label{eq:Hlog1}
\end{equation}
By Lemma \ref{lem:coeff2} and \eqref{eq:Hlog1}, a direct computation shows that
\begin{equation}
\begin{aligned}
H_{2,1}(F_f/2) &= \frac{1}{432}\bigl(-48\tau_1^4 + 64(1-\tau_1^2)\tau_1^2\tau_2 \\
&\quad -64(1-\tau_1^2)(3+\tau_1^2)\tau_2^2 + 256(1-\tau_1^2)(1-|\tau_2|^2)\tau_1\tau_3\bigr).
\end{aligned} \label{eq:Hlog2}
\end{equation}
Since the class \(\mathcal{P}\) and \(H_{2,1}(F_f/2)\) are invariant under rotation, we may assume that \(c_1\in [0,2]\) (see \cite[Theorem 3]{ref12}), that is in view of \eqref{eq:c1}, \(\tau_1\in [0,1]\).

Consider the following cases:

\textbf{Case 1.} Suppose that \(\tau_1=1\). Then from \eqref{eq:Hlog2}, we obtain
\[
|H_{2,1}(F_f/2)| = |\gamma_1 \gamma_3 - \gamma_2^2| = \frac{48}{432} = \frac{1}{9}.
\]

\textbf{Case 2.} Suppose that \(\tau_1=0\). Then from \eqref{eq:Hlog2}, we obtain
\[
|H_{2,1}(F_f/2)| = |\gamma_1 \gamma_3 - \gamma_2^2| = \frac{64}{432} |3\tau_2^2| = \frac{4}{9}.
\]

\textbf{Case 3.} Suppose that \(\tau_1\in (0,1)\). Then from \eqref{eq:Hlog2} we obtain
\[
\begin{aligned}
|H_{2,1}(F_f/2)| &\leq \frac{1}{432}\bigl(|48\tau_1^4 - 64(1-\tau_1^2)\tau_1^2\tau_2 \\
&\quad +64(1-\tau_1^2)(3+\tau_1^2)\tau_2^2| + 256(1-\tau_1^2)(1-|\tau_2|^2)\tau_1\tau_3\bigr)\\
&= \frac{16}{27}\tau_1(1-\tau_1^2)(|A+B\tau_2+C\tau_2^2| + 1-|\tau_2|^2),
\end{aligned}
\]
where
\[
A = \frac{3\tau_1^3}{16(1-\tau_1^2)}, \quad B = -\frac{1}{4}\tau_1, \quad C = \frac{3+\tau_1^2}{4\tau_1}.
\]
It follows that \(AC>0\). Moreover, the relation
\[
|B| - 2(1-|C|) = \frac{1}{4}\tau_1 + \frac{3+\tau_1^2}{4\tau_1} - 2 = \frac{3\tau_1^2+4}{4\tau_1} >0
\]
holds. 
Applying Lemma \ref{lem:max} in the above inequality, we obtain
\[
|H_{2,1}(F_f/2)| \leq \frac{16}{27}\tau_1(1-\tau_1^2)(|A|+|B|+|C|)
= \frac{4}{27}(-\tau_1^4+2\tau_1^2+3) = g(\tau_1^2).
\]
Set \(t = \tau_1^2 \in (0,1)\). Then \(g(t)=-t^2+2t+3\) and so \(g'(t)=2(1-t)>0\). Therefore \(g\) is an increasing function on \((0,1)\). Hence \(\max_{t\in (0,1)}\{g(t)\}=4\).

From the above equation we deduce that
\[
|H_{2,1}(F_f/2)| \leq \frac{16}{27}.
\]

To verify that the bound \(|H_{2,1}(F_f/2)|\leq \frac{16}{27}\) is sharp, consider the function \(f_1:\mathbb{D}\to \mathbb{C}\) defined by
\[
f_1(z) = z\exp\left(\frac{z^4}{6} + \frac{2}{3} z^2\right) = z + \frac{2}{3} z^3 + \frac{7}{18} z^5 + \cdots.
\]
We observe that
\[
\frac{z f_1'(z)}{f_1(z)} = 1 + \frac{4}{3} z^2 + \frac{2}{3} z^4.
\]
Thus, \(f_1\in \mathcal{S}_{car}^*\). In this case, \(a_2=0\), \(a_3=\frac{2}{3}\) and \(a_4=0\), which establishes the sharpness of the bound and completes the proof of the theorem.
\end{proof}

\begin{theo}
Let \(f\in \mathcal{S}_{car}^*\) be given by \eqref{eq:series}. Then
\[
\left|H_{2,1}(F_{f^{-1}}/2)\right| \leq \frac{71}{324}.
\]
The inequality is sharp.
\end{theo}

\begin{proof}
Since \(f\in \mathcal{S}_{car}^*\), utilizing relations \eqref{eq:invlogcoeff} and \eqref{eq:coeff1}, we find
\begin{equation}
\begin{aligned}
H_{2,1}(F_{f^{-1}}/2) &= \frac{1}{48}\left(16a_2^4 - 12a_2^2a_3 - 12a_3^2 + 12a_2a_4\right)\\
&= \frac{1}{48}\left(\frac{167}{108}c_1^4 - 2c_1^2c_2 + \frac{16}{9}c_1c_3 - \frac{4}{3}c_2^2\right)\\
&= \frac{1}{5184}\left(167c_1^4 - 216c_1^2c_2 + 192c_1c_3 - 144c_2^2\right).
\end{aligned} \label{eq:Hinv1}
\end{equation}
Substituting the coefficient representations from Lemma \ref{lem:coeff2} into \eqref{eq:Hinv1}, a direct computation yields
\begin{equation}
\begin{aligned}
H_{2,1}(F_{f^{-1}}/2) &= \frac{1}{5184}\bigl(1136\tau_1^4 - 1344(1-\tau_1^2)\tau_1^2\tau_2 \\
&\quad -192(1-\tau_1^2)(3+\tau_1^2)\tau_2^2 + 768(1-\tau_1^2)(1-|\tau_2|^2)\tau_1\tau_3\bigr)\\
&= \frac{1}{324}\bigl(71\tau_1^4 - 84(1-\tau_1^2)\tau_1^2\tau_2 \\
&\quad -12(1-\tau_1^2)(3+\tau_1^2)\tau_2^2 + 48(1-\tau_1^2)(1-|\tau_2|^2)\tau_1\tau_3\bigr).
\end{aligned} \label{eq:Hinv2}
\end{equation}
Since the class \(\mathcal{P}\) and \(H_{2,1}(F_{f^{-1}}/2)\) are invariant under rotation, we may assume that \(c_1\in [0,2]\) (see \cite[Theorem 3]{ref12}), which implies \(\tau_1\in [0,1]\) in view of \eqref{eq:c1}.

Consider the following three cases:

\textbf{Case 1.} Suppose that \(\tau_1=1\). Then from \eqref{eq:Hinv2}, we obtain
\[
|H_{2,1}(F_{f^{-1}}/2)| = |\Gamma_1\Gamma_3 - \Gamma_2^2| = \frac{71}{324}.
\]

\textbf{Case 2.} Suppose that \(\tau_1=0\). Then from \eqref{eq:Hinv2}, we obtain
\[
|H_{2,1}(F_{f^{-1}}/2)| = |\Gamma_1\Gamma_3 - \Gamma_2^2| = \frac{12}{324}|3\tau_2^2| = \frac{1}{9}.
\]

\textbf{Case 3.} Suppose that \(\tau_1\in (0,1)\). Then from \eqref{eq:Hinv2} we obtain
\begin{equation}
\begin{aligned}
|H_{2,1}(F_{f^{-1}}/2)| &\leq \frac{1}{324}\bigl(|71\tau_1^4 - 84(1-\tau_1^2)\tau_1^2\tau_2 \\
&\quad -12(1-\tau_1^2)(3+\tau_1^2)\tau_2^2| + 48(1-\tau_1^2)(1-|\tau_2|^2)\tau_1\tau_3\bigr)\\
&= \frac{4}{27}\tau_1(1-\tau_1^2)(|A+B\tau_2+C\tau_2^2| + 1-|\tau_2|^2),
\end{aligned} \label{eq:Hinv3}
\end{equation}
where
\[
A = \frac{71\tau_1^3}{48(1-\tau_1^2)}, \quad B = -\frac{7}{4}\tau_1, \quad C = -\frac{3+\tau_1^2}{4\tau_1}.
\]
Observe that \(AC<0\). Hence, we apply case (ii) of Lemma \ref{lem:max}. We now verify each parameter condition appearing in case (ii) of Lemma \ref{lem:max}.

\begin{enumerate}
\item[(a)] Note that the relation
\[
-4AC\left(\frac{1}{C^2}-1\right) - B^2 = \frac{71\tau_1^2(3+\tau_1^2)}{48(1-\tau_1^2)}\left(\frac{16\tau_1^2}{(3+\tau_1^2)^2}-1\right) - \frac{49}{16}\tau_1^2 \leq 0
\]
is equivalent to
\[
\frac{76\tau_1^4 + 1004\tau_1^2 - 1080}{48(3+\tau_1^2)(1-\tau_1^2)} \leq 0.
\]
Factoring the numerator yields \(4(\tau_1^2-1)(19\tau_1^2+270)\). Since \(\tau_1\in (0,1)\), this expression is strictly negative, confirming that the condition \(-4AC(\frac{1}{C^2}-1) \leq B^2\) holds for all \(\tau_1\in (0,1)\). However, the simultaneous condition \(|B|<2(1-|C|)\) simplifies to \(9\tau_1^2 - 8\tau_1 + 6 < 0\), which has no real roots and is false for all \(\tau_1\in (0,1)\). This excludes the first subcase branch of case (ii) of Lemma \ref{lem:max}.

\item[(b)] We evaluate the strict inequality criterion \(B^2 < \min\{4(1+|C|)^2, -4AC(\frac{1}{C^2}-1)\}\). Since \(-4AC(\frac{1}{C^2}-1)<0\) holds for all \(\tau_1\in (0,1)\), the minimum value of the set is negative. Because \(B^2 = \frac{49}{16}\tau_1^2 \geq 0\), the strict inequality is false for all \(\tau_1\in (0,1)\), excluding the second subcase branch of case (ii).

\item[(c)] Checking the first subcase of the remaining solution branch \(R(A,B,C)\), the condition \(|C|(|B|+4|A|) - |AB| \leq 0\) equates to \(84+228\tau_1^2-99\tau_1^4 \leq 0\), which is false for all \(\tau_1\in (0,1)\).

\item[(d)] The second subcase condition of \(R(A,B,C)\), given by \(|AB| - |C|(|B|-4|A|) \leq 0\), is equivalent to \(865\tau_1^4 + 1020\tau_1^2 - 252 \leq 0\). Setting \(x=\tau_1^2\), this translates to the quadratic inequality \(865x^2 + 1020x - 252 \leq 0\). The unique positive root of the corresponding equation is \(x_0 = \frac{-1020 + \sqrt{1912320}}{1730} \approx 0.2105\). Hence, the inequality holds true within the parameter sub-interval \(0< \tau_1 \leq \tau_1' = \sqrt{x_0} \approx 0.458\)
\end{enumerate}

Applying the corresponding branch of Lemma \ref{lem:max} for \(0< \tau_1 \leq \tau_1'\), we obtain:
\begin{equation}
\begin{aligned}
|H_{2,1}(F_{f^{-1}}/2)| &\leq \frac{4}{27}\tau_1(1-\tau_1^2)(-|A|+|B|+|C|)\\
&= \frac{1}{324}(36+60\tau_1^2-167\tau_1^4) = \phi(\tau_1).
\end{aligned} \label{eq:Hinv4}
\end{equation}
Differentiating with respect to \(\tau_1\) yields \(\phi'(\tau_1) = \frac{1}{324}(120\tau_1 - 668\tau_1^3)\). Setting \(\phi'(\tau_1)=0\) reveals an interior critical point at \(\tau_1 = \sqrt{\frac{30}{167}} \approx 0.4237 \in (0,\tau_1']\). Evaluating the functional values reveals that \(\phi(\tau_1)\) increases on \((0,\sqrt{\frac{30}{167}})\) and decreases on \((\sqrt{\frac{30}{167}},\tau_1')\), yielding a maximum value of:
\[
\phi\left(\sqrt{\frac{30}{167}}\right) = \frac{1782}{13933} \approx 0.1279.
\]

For the remaining parameter range \(\tau_1' < \tau_1 < 1\), the final branch of \(R(A,B,C)\) applies, giving:
\begin{equation}
\begin{aligned}
|H_{2,1}(F_{f^{-1}}/2)| &\leq \frac{4}{27}\tau_1(1-\tau_1^2)(|C|+|A|)\sqrt{1-\frac{B^2}{4AC}}\\
&= \frac{1}{324}(36-24\tau_1^2+59\tau_1^4)\sqrt{\frac{360-76\tau_1^2}{71(3+\tau_1^2)}} = \psi(\tau_1).
\end{aligned} \label{eq:Hinv5}
\end{equation}
Let \(t=\tau_1^2\in (\tau_1'^2,1)\). Squaring \(\psi(\tau_1)\) to evaluate its monotonic properties analytically via \(g(t)=\psi^2(\tau_1)\) yields:
\[
g(t) = \frac{1}{324^2} \frac{(36-24t+59t^2)^2(360-76t)}{3+t}.
\]
Differentiating with respect to \(t\) yields:
\[
g'(t) = \frac{(59t^2-24t+36)\Delta(t)}{324^2(3+t)^2},
\]
where \(\Delta(t) = -17936t^3 + 108t^2 + 262656t - 73008\). To rigorously determine the sign of \(g'(t)\), we observe that its derivative \(\Delta'(t) = -53808t^2 + 216t + 262656\) satisfies \(\Delta'(t)>0\) for all \(t\in (0,1)\), meaning \(\Delta(t)\) is strictly increasing on \((0,1)\). Since \(\Delta(0)=-73008<0\) and \(\Delta(1)=171820>0\), the Intermediate Value Theorem guarantees that \(\Delta(t)\) possesses a unique real root \(t_1\in (0,1)\), which evaluates to \(t_1\approx 0.2794\).

Since \(\tau_1'^2 \approx 0.2105 < t_1\), it follows that \(g'(t)<0\) on \((\tau_1'^2,t_1)\) and \(g'(t)>0\) on \((t_1,1)\). Thus, \(g(t)\) is strictly decreasing on \((\tau_1'^2,t_1)\) and strictly increasing on \((t_1,1)\). Evaluating \(\psi(\tau_1)=\sqrt{g(t)}\) at the endpoints of this interval yields:
\[
\psi(\tau_1') \approx 0.1273 \quad \text{and} \quad \psi(1) = \frac{71}{324} \approx 0.2191.
\]
Since \(\frac{1}{9} < \frac{1782}{13933} < \frac{71}{324}\), compiling the upper bounds across Case 1 (\(\frac{71}{324}\)), Case 2 (\(\frac{1}{9}\)), and Case 3 (\(\max\{\frac{1782}{13933},0.1273,\frac{71}{324}\}\)) yields:
\[
|H_{2,1}(F_{f^{-1}}/2)| \leq \max\left\{\frac{1}{9}, \frac{1782}{13933}, \frac{71}{324}\right\} = \frac{71}{324}.
\]

To verify that the result is sharp, consider the function \(f_2(z) = z + \frac{4}{3} z^2 + \frac{4}{3} z^3 + \frac{68}{81} z^4 + \cdots\). For this function, the initial Taylor coefficients are \(a_2=\frac{4}{3}\), \(a_3=\frac{4}{3}\), and \(a_4=\frac{68}{81}\). Substituting these expressions directly into the definitions for the logarithmic inverse coefficients given by \eqref{eq:invlogcoeff}, we obtain:
\[
\Gamma_1 = -\frac{2}{3}, \quad \Gamma_2 = \frac{2}{3}, \quad \text{and} \quad \Gamma_3 = -\frac{29}{162}.
\]
Computing the second Hankel functional explicitly under these values yields:
\[
H_{2,1}(F_{f^{-1}}/2) = \Gamma_1\Gamma_3 - \Gamma_2^2 = \left(-\frac{2}{3}\right)\left(-\frac{29}{162}\right) - \left(\frac{2}{3}\right)^2 = \frac{58}{486} - \frac{4}{9} = -\frac{71}{324},
\]
which gives \(|H_{2,1}(F_{f^{-1}}/2)| = \frac{71}{324}\). This establishes the sharpness of the estimate and completes the proof.
\end{proof}

\section{Hermitian--Toeplitz determinant}

In this section, we provide sharp bounds on the third-order Hermitian--Toeplitz determinant for the class \(\mathcal{S}_{car}^*\).

\begin{theo}
Let \(f\in \mathcal{S}_{car}^*\) be given by \eqref{eq:series}. Then
\[
-\frac{841}{2295} \leq T_{3,1}(f) \leq 1.
\]
The inequalities are sharp.
\end{theo}

\begin{proof}
Since \(f\in \mathcal{S}_{car}^*\), proceeding in a similar manner as in Theorem 3.1, we obtain \(a_2\) and \(a_3\) from \eqref{eq:coeff1}.

It is important to note that both the class \(\mathcal{P}\) of functions with positive real part and the class \(f\in \mathcal{S}_{car}^*\) are invariant under rotations. Hence, without loss of generality, since \(|c_n|\leq 2\), we may assume \(0\leq c_1\leq 2\).

In view of Lemma \ref{lem:coeff1} together with \eqref{eq:coeff1}, we obtain, for some \(\xi \in \overline{\mathbb{D}}\), that
\begin{equation}
\begin{aligned}
2\Re(a_2^2 \overline{a_3}) &= 2\Re\left(\frac{4c_1^2}{9}\right)\left(\frac{1}{3}\overline{c}_2 + \frac{5}{36}\overline{c}_1^2\right)\\
&= \frac{8c_1^2}{9}\left(\frac{5}{36}c_1^2 + \frac{1}{6}(c_1^2+(4-c_1^2)\Re(\overline{\xi}))\right)\\
&= \frac{22c_1^4}{81} + \frac{4c_1^2(4-c_1^2)\Re\overline{\xi}}{27},
\end{aligned} \label{eq:HT1}
\end{equation}
\[
-2|a_2|^2 = -\frac{8}{9}c_1^2, \label{eq:HT2}
\]
and
\begin{equation}
\begin{aligned}
-|a_3|^2 &= -\left|\frac{1}{3}c_2 + \frac{5}{36}c_1^2\right|^2\\
&= -\left|\frac{1}{6}c_1^2 + \frac{1}{6}(4-c_1^2)\xi + \frac{5}{36}c_1^2\right|^2\\
&= -\left|\frac{11}{36}c_1^2 + \frac{1}{6}(4-c_1^2)\xi\right|^2\\
&= -\frac{121}{1296}c_1^4 - \frac{11}{108}c_1^2(4-c_1^2)\Re\xi - \frac{1}{36}(4-c_1^2)^2|\xi|^2.
\end{aligned} \label{eq:HT3}
\end{equation}

By applying equations \eqref{eq:HT1}, \eqref{eq:HT2}, and \eqref{eq:HT3}, equation \eqref{eq:HT31} can be written as
\begin{equation}
\begin{aligned}
T_{3,1}(f) &= 1 + \frac{22c_1^4}{81} + \frac{4c_1^2(4-c_1^2)\Re\overline{\xi}}{27} - \frac{8}{9}c_1^2 \\
&\quad -\frac{121}{1296}c_1^4 - \frac{11}{108}c_1^2(4-c_1^2)\Re\xi - \frac{1}{36}(4-c_1^2)^2|\xi|^2\\
&= \frac{1}{1296}\left(1296 + 231c_1^4 - 1152c_1^2 + 60c_1^2(4-c_1^2)\Re\xi - 36(4-c_1^2)^2|\xi|^2\right).
\end{aligned} \label{eq:HT4}
\end{equation}

We determine the maximum of the right-hand side of \eqref{eq:HT4}. Since \(\Re\xi \leq |\xi|\), it follows from \eqref{eq:HT4} that
\begin{equation}
\begin{aligned}
T_{3,1}(f) &\leq \frac{1}{1296}\left(1296 + 231c_1^4 - 1152c_1^2 + 60c_1^2(4-c_1^2)|\xi| - 36(4-c_1^2)^2|\xi|^2\right)\\
&= \frac{1}{1296} F(c_1^2,|\xi|).
\end{aligned} \label{eq:HT5}
\end{equation}

Setting \(c_1^2 =: x\in [0,4]\) and \(|\xi|=:y\in [0,1]\), then \(F(c_1^2,|\xi|)\) can be written as follows
\begin{equation}
F(x,y) = 1296 + 231x^2 - 1152x + 60x(4-x)y - 36(4-x)^2 y^2. \label{eq:F}
\end{equation}
Now, differentiating partially \eqref{eq:F} with respect to \(x\) and \(y\), we obtain
\[
\frac{\partial F(x,y)}{\partial x} = 462x - 120xy - 72xy^2 + 240y + 288y^2 - 1152
\]
and
\[
\frac{\partial F(x,y)}{\partial y} = 240x - 60x^2 - 1152y + 576xy - 72x^2 y.
\]
Solving \(\frac{\partial F(x,y)}{\partial x}=0\) and \(\frac{\partial F(x,y)}{\partial y}=0\), we find that the only critical points are \((4,\frac{29}{10})\) and \((\frac{9}{4},\frac{15}{14})\). Since \(y=\frac{29}{10}>1\) and \(y=\frac{15}{14}>1\), both points lie strictly outside the standard rectangular domain \([0,4]\times [0,1]\). Therefore, the maximum value of \(F(x,y)\) is attained exclusively on the boundary of \([0,4]\times [0,1]\).

On the boundary of the rectangular region \([0,4]\times [0,1]\), the function \(F(x,y)\) takes the following forms:
\[
F(0,y) = 1296 - 144y^2 \leq 1296, \quad F(4,y)=384 \text{ for all } y\in [0,1]
\]
and
\[
F(x,0) = 1296 + 231x^2 - 1152x, \quad F(x,1) = 3(45x^2 - 208x + 240) \text{ for all } x\in [0,4].
\]
It follows that \(\max_{x\in [0,4]}F(x,0)=1296\) and \(\max_{x\in [0,4]}F(x,1)=720\).

From the above discussion, we deduce that
\[
T_{3,1}(f) \leq \frac{1}{1296}\max\{1296,384,720\} = 1.
\]

Next, we determine the minimum of the right-hand side of \eqref{eq:HT4}. Since \(-\Re\xi \leq |\xi|\), it follows from \eqref{eq:HT4} that
\begin{align}
T_{3,1}(f) &\geq \frac{1}{1296}\left(1296 + 231c_1^4 - 1152c_1^2 - 60c_1^2(4-c_1^2)|\xi| - 36(4-c_1^2)^2|\xi|^2\right) \notag\\
&\geq \frac{1}{1296} G(c_1^2,1). \label{eq:HT6}
\end{align}
Setting \(c_1^2=:x\in [0,4]\), then \(G(c_1^2,1)\) can be written as follows
\[
G(x,1) = 1296 + 231x^2 - 1152x - 60x(4-x) - 36(4-x)^2 = 255x^2 - 1104x + 720.
\]
Here, we observe that \(G'(x,1)=0\) for \(x=\frac{184}{85}\). Moreover, since \(G''(x,1)=510>0\) it follows that \(G(x,1)\) attains its minimum at \(x_0=\frac{184}{85}\), where
\[
G(x_0,1) = -\frac{40368}{85}.
\]
From the above discussion, we deduce that
\[
T_{3,1}(f) \geq -\frac{841}{2295}.
\]
This completes the proof.

To verify that the upper bound of \(T_{3,1}(f)\) is sharp, consider the function \(p_2:\mathbb{D}\to \mathbb{C}\) defined by
\[
p_2(z) \coloneqq \frac{1+z^3}{1-z^3}, \qquad z\in \mathbb{D}.
\]
The equality holds for the function \(f_3\) defined by
\[
\frac{z f_3'(z)}{f_3(z)} = 1 + \frac{4}{3}\frac{p_2(z)-1}{p_2(z)+1} + \frac{2}{3}\left(\frac{p_2(z)-1}{p_2(z)+1}\right)^2,
\]
which yields
\[
f_3(z) = z + \frac{4}{9} z^3 + \cdots.
\]

To verify that the lower bound of \(T_{3,1}(f)\) is sharp, consider the function \(p_3:\mathbb{D}\to \mathbb{C}\) defined by
\[
p_3(z) \coloneqq \frac{1-z^2}{1-\sqrt{\frac{184}{85}}z+z^2}, \qquad z\in \mathbb{D}.
\]
The equality holds for the function \(f_4\) defined by
\[
\frac{z f_4'(z)}{f_4(z)} = 1 + \frac{4}{3}\frac{p_3(z)-1}{p_3(z)+1} + \frac{2}{3}\left(\frac{p_3(z)-1}{p_3(z)+1}\right)^2,
\]
which yields
\[
f_4(z) = z + \frac{2}{3}\sqrt{\frac{184}{85}} z^2 + \cdots.
\]
\end{proof}

\section*{{\bf Declarations}}
\subsection*{Data Availability Statement}
Data sharing is not applicable to this article as no datasets were generated or analyzed during the current study.
\subsection*{Conflict of Interest}
The authors declare that they have no conflict of interest. 
\subsection*{Author Contributions}
Both authors contributed equally to this work.

\end{document}